\documentclass[11pt]{article}

\usepackage{graphicx}
\usepackage{color}
\usepackage[all]{xy}
\usepackage{amssymb}

\def\R {{\rm I\! R}}

\newcommand{\dst}{\displaystyle}
\newtheorem{defn}{Definition}[section]

\newtheorem{prop}[defn]{Proposition}
\newtheorem{thm}[defn]{Theorem}
\newtheorem{cor}[defn]{Corollary}

\newtheorem{rem}[defn]{Remark}

\def\dst{\displaystyle}

\def\square{\ifmmode\sqr\else{$\sqr$}\fi}
\def\sqr{\vcenter{
         \hrule height.1mm
         \hbox{\vrule width.1mm height2.2mm\kern2.18mm
\vrule width.1mm}
         \hrule height.1mm}}

\newenvironment{proof}[1]{
  \trivlist \item[\hskip \labelsep{\it #1}]}{\hfill\mbox{$\square$}
  \endtrivlist}
\begin{document}

\title{Some relationships between the geometry of the tangent bundle and the geometry of the Riemannian base manifold.}

\author{{Guillermo  Henry and Guillermo Keilhauer}\\Departamento de Matem\'atica, FCEyN,\\
 Universidad de Buenos Aires, Argentina.}

\date{}

\maketitle


{\bf{Abstract.}} We compute the curvature tensor of the tangent bundle of a Riemannian manifold endowed with a natural metric and we  get some relationships between the geometry of the base manifold and the geometry of the tangent bundle.

\vspace{0.5 cm}


{\bf{Mathematics Subject Clasification (2000):\  }}  53C20 $\cdot$  53B21 $\cdot$  53A55
\section{Introduction}

Let $(M,g)$ be a Riemannian manifold of dimension $n\geq 2$. Let $\pi:TM\longrightarrow M$ and $P:O(M)\longrightarrow M$ be  the tangent and the orthonormal bundle over $M$ respectively.
 In this paper we deal with a class of Riemannian metrics $G$  on $TM$. These metrics  makes  $\pi:(TM,G)\longrightarrow (M,g)$ a Riemannian submersion,  the horizontal distribution induced by the Levi-Civita connection of $(M,g)$  orthogonal to the vertical distribution and  $G$ is   the image by a natural operator of order two of the metric $g$. The  Sasaki metric  and the Cheeger-Gromoll  metric are well known examples of these class of metrics, and there were extensively studied by Kowalski \cite{Ko}, Aso \cite{KA}, Sekizawa \cite{Se},  Musso and Tricerri \cite{MT}, Gudmundsson and Kappos \cite{GKap} among others. The notion of {\it{natural tensor}} on the tangent bundle of a Riemannian manifold as a tensor that is the image  by a natural operator of order two of the base manifold metric, was introduced and characterized by Kowalski and Sekizawa in \cite{Ko-Se}.  In \cite{MCC-GK},  Calvo and Keilhauer showed that for a  given  Riemannian manifold $(M,g)$, any $(0,2)$ tensor field on $TM$ admits a global matrix representation. Using this one to one relationship, they defined and characterized, without making use of the theory of differential invariants, what they also called $natural\ tensor$. In the symmetric case this concept coincide with the one of Kowalski and Sekizawa. In \cite{GH}, the first author gives a new approach of the concept of naturality, introducing the notion of \textit{s-space} and  $\lambda$-\textit{naturality}. This approach avoids jets and natural operators theory and  generalized the one given in \cite{MCC-GK} and \cite{Ko-Se}.

 In section 2, we introduce natural metrics on $TM$ by means of \cite{MCC-GK}. For any $q\in M$, let $M_q$ be the tangent space of $M$ at $q$. Let $\psi:N:O(M)\times\R^{n}\longrightarrow TM$ be the projection defined by

 \begin{equation}\label{projectionpsi}
 \psi(q,u,\xi)=\dst\sum_{i=1}^n\xi^iu_i
 \end{equation}
where $u=(u_1,\dots,u_n)$ is an orthonormal basis for $M_q$ and $\xi=(\xi^1,\dots,\xi^n)\in \R^n$. It is well known (see \cite{MT}), that for a fixed Riemannian metric on $TM$ a suitable Riemannian metric $G^*$ on $N$ can be defined such that $\psi:(N,G^*)\longrightarrow (TM, G)$ is a Riemannian submersion. Based on this fact and the O'Neill formula, in Section \ref{section3}, we compute the curvature tensor of $(TM,G)$, when $G$ is a natural metric. As an application, we get in Section 4 some relationships between the geometry of $TM$ and the geometry of $M$.

Throughout, all geometric objets are assumed to be differentiable, i.e. $C^{\infty}$.


\section{Preliminaries.}

Let $\nabla$ be the Levi-Civita connection of $g$ and $K:TTM\longrightarrow TM$ the connection map induced by $\nabla$. For any $q\in M$ and $v\in M_q$, let $\pi_{*_v}:(TM)_v\longrightarrow M_q$ be the differential map of $\pi$  at $v$, and $K_v:(TM)_v\longrightarrow M_q$ the restriction of $K$ to $(TM)_v$.

Since the linear map $\pi_{*_v}\times K_v:(TM)_v\longrightarrow M_q\times M_q$ defined by $(\pi_{*_v}\times K_v)(b)=(\pi_{*_v}(b),K_v(b))$ is an isomorphism that maps the horizontal subspace $(TM)^h_v=\ker K_v$ onto $M_q\times \{0_q\}$ and the vertical subspace $(TM)^v_v=\ker \pi_{*_v}$ onto $\{0_q\}\times M_q$, where $0_q$ denotes the zero vector, we define differentiable mappings $e_i,e_{n+i}:N=O(M)\times \R^n\longrightarrow TTM$ for $i=1,\dots, n$ and $v=\psi(q,u,\xi)$ by
\begin{eqnarray}
e_i(q,u,\xi)=(\pi_{*_v}\times K_v)^{-1}(u_i,0_q)\nonumber\\
\\
 e_{n+i}(q,u,\xi)=(\pi_{*_v}\times K_v)^{-1}(0_q,u_i)\nonumber
\end{eqnarray}
 The action of the orthonormal group $O(n)$ of $\R^{n\times n}$ on $N$ is given by the family of maps $R_a:N\longrightarrow N$, $a\in O(n)$,  $R_a(q,u,\xi)=(q,u.a,\xi.a)$ where $u.a=(\sum_{i=1}^n a_1^i
u_i,\ldots,\sum_{i=1}^n a_n^i u_i)$ and $\xi.a=(\sum_{i=1}^n a_1^i
\xi^i,\ldots,\sum_{i=1}^n a_n^i \xi^i)$. It is  easy to see that $$\{e_i(R_a(p,u,\xi))\}=\{e_l(p,u,\xi)\}.L(a)$$ where $L:O(n)\longrightarrow \R^{2n\times 2n}$ is the map defined by
\begin{equation}L(a)=\pmatrix{a & 0\cr 0 & a}
\end{equation}

For any $(0,2)$ tensor field  $T$ on $TM$ we define the differentiable function $^gT:N\longrightarrow\R^{2n\times 2n}$ as follows: If $(q,u,\xi)\in N$ and $v=\psi(q,u,\xi)$, let $^gT(q,u,\xi)$ be the matrix of the bilinear form $T_v:(TM)_v\times(TM)_v\longrightarrow\R$ induced by $T$ on $(TM)_v$ with respect to the basis $\{e_1(q,u,\xi),\dots,e_{2n}(q,u,\xi)\}$. One sees easily that $^gT$ satisfies the following invariance property:
\begin{equation}\label{invarianze property}
^gT\circ R_a=(L(a))^t.^gT.L(a)
\end{equation}
Moreover, there is a one to one correspondence between the $(0,2)$  tensor fields on $TM$ and  differentiable maps $^gT$ satisfying (\ref{invarianze property}).

A tensor field  $T$  on $TM$ will be call natural with respect to $g$ if $^gT$ depends only of the parameter $\xi$, (see \cite{MCC-GK}). In the sense of \cite{GH}, the collection  $\lambda=(N, \psi, O(n),\tilde{R}, \{e_i\})$  is a s-space over $TM$,  with base change morphism  $L$; and the natural tensors with respect to $g$ are the $\lambda-natural$ tensors with respect to $TM$.

In this paper we will call $G$ a natural metric on $TM$ if:
\begin{itemize}
\item[1.] $G$ is a Riemannian metric such that $\pi:(TM,G)\longrightarrow (M,g)$ is a Riemannian submersion.
\item[2.] For $v\in TM$, the subspaces $(TM)^v_v$ and $(TM)^h_v$ are orthogonals.
\item[3.] $G$ is natural with respect to $g$.
\end{itemize}
From Lemma $3.1$ of \cite{MCC-GK}, it follows that $G$ is a natural metric on $TM$ if

\begin{equation}\label{caracterizacionmetricanatural}
^{g} G(p,u,\xi)=\pmatrix{Id_{n\times n} & 0 \cr 0
&\alpha(\|\xi\|^2).Id_{n\times n}+\beta(\|\xi\|^2)(\xi)^t.\xi}
\end{equation}
where  $\alpha , \beta:[0,+ \infty )\longrightarrow \R
$ are differentiable functions  satisfying  $\alpha (t)> 0$, and $\alpha(t)+t\beta(t)> 0$ for all $t\geq 0$.

\begin{rem} The Sasaki metric $G_s$ corresponds to the case $\alpha=1,\ \beta=0$; and the Cheeger-Gromoll metric $G_{ch}$ to the case $\alpha=\beta$, and $\alpha(t)=\frac{1}{1+t}$.
\end{rem}


\section{Curvature equations.}\label{section3}
In this section we compute the curvature tensor of $TM$ endowed with a natural metric. Since this computation involves well known objects defined on $N$, we shall begin to describe them briefly using the connection map.

\subsection{Canonical constructions on $N$.}

Let $\theta^i$, $\omega^i_j$ be the canonical 1-forms on $O(M)$, which in terms of the connection map are defined as follows:

\begin{equation}
\theta^i(q,u)(b)=g_q\Big(P_{*_{(q,u)}}(b),u_i\Big)
\end{equation}
  \begin{equation}
 \omega^i_j(q,u)(b)=g_q\Big(K((\pi_j)_{*_{(q,u)}}(b)),u_i\Big)
  \end{equation}

 where $\pi_j:O(M)\longrightarrow TM$ is the $j^{th}$ projection, i.e. $\pi_j(q,u)=u_j$ and $1\leq i,j\leq n$.

 From now on, let $\theta^i$, $\omega^i_j$, $d\xi^i$ be the pull backs of the canonical 1-forms and the usual 1-forms on $\R ^n$ by $P_1:N\longrightarrow O(M)$ and $P_2:N\longrightarrow \R^n$.

 For any $z\in N$ let us denote by $V_z=\ker \psi_{*_z}$ and $H_z=\{b\in N_z:\omega^i_j(z)(b)=0,1\leq i<j\leq n\}$  the vertical and the horizontal subspace of $N_z$ respectively. By letting \cite{MT}
 \begin{equation}
 \theta^{n+i}=d\xi^i + \sum_{j=1}^n\xi^j.\omega^i_j
 \end{equation}
 we get that for any $z\in N$, $\{\theta^1(z),\dots,\theta^{2n}(z),\{\omega^i_j(z)\}\}$ is a basis for $N^*_z$ and $V_z=\{b\in N_z:\theta^l(z)(b)=0 \ \mbox{for}\ 1\leq l\leq 2n \}$.

Let  $H_1,\dots,H_{2n},\{V^l_m\}_{1\leq l<m\leq n}$ be the dual frame of $\{\theta^1,\dots,\theta^{2n},\{\omega^i_j\}\}$. The vector fields were constructed as follows:  If $z=(q,u,\xi)$, let $c_i$ be the geodesic that satisfies $c_i(0)=q$ and $\dot{c}_i(0)=u_i$. Let $E_1^i,\dots,E_n^i$ be the parallel vector fields along $c_i$ such that $E^i_l(0)=u_l$. If we define   $\gamma_i(t)=(c_i(t),E^i_1(t),\ldots,E^i_n(t),\xi)$, then
 \begin{equation}
H_i(z)=\dot{\gamma}_i(z)
\end{equation}
\begin{equation}
H_{n+i}(z)=(i_{(q,u)})_{*_\xi}(\frac{\partial}{\partial
\xi^i}|_{\xi})
\end{equation}

 for $1\leq i\leq n$, where $i_{(q,u)}:\R^n\longrightarrow N$ is the inclusion map given by $i_{(q,u)}(\xi)=(q,u,\xi)$.

Let $\sigma_z:O(n)\longrightarrow N$ be the map defined by $\sigma_z(a)=R_a(z)=z.a$. Since $V_z=\ker(\psi_{*_z})=
 (\sigma_z)_{*_{Id}}({\mathfrak{o}}(n))$, where ${\mathfrak{o}}$ is the space of skew symmetric matrices of $R^{n\times n}$, let
\begin{equation}
V^l_m(z)=(\sigma_z)_{*_{id}}(A^l_m)
\end{equation}
where $[A^l_m]^l_m=1$, $[A^l_m]^m_l=-1$ and $[A^l_m]^i_j=0$ otherwise. Hence,
\begin{equation}\label{Hiypsi1}
\psi_{*_z}(V^l_m(z))=0
\end{equation}
 An easy check shows that
\begin{equation}\label{Hiypsi2}
\psi_{*_z}(H_i(z))=e_i(z)
\end{equation}
and
\begin{equation}\label{Hiypsi3}
\psi_{*_z}(H_{n+i}(z))=e_{n+i}(z)
\end{equation}

  Let  $\omega=\sum_{1\leq i<j\leq n} \omega^i_j\otimes \omega^i_j$, if $G$ is a Riemannian metric on $TM$  then
\begin{equation}\label{9}
G^*=\psi^*(G)+\omega
\end{equation}

is also a Riemannian metric on $N$. It follows easily  that  $(V_z)\perp_{G^*}H_z$ and $\psi_{*_{z}}:H_{z}\longrightarrow (TM)_{\psi(z)}$  is an isometry, therefore $\psi:(N,G^*)\longrightarrow (TM,G)$ is a Riemannian submersion. We shall use this fact to compute the curvature tensor of $(TM,G)$ when $G$ is a natural metric.

\begin{rem} 
 Let $X$ be a vector field on $TM$, the horizontal lift of $X$ is a vector field $X^h$ on $N$ such that $X^h(z)\in H_z$ and $\psi_{*_z}(X^h(z))=X(\psi(z))$. If $X(\psi(z))=\sum_{i=1}^{2n}x^i(z)e_i(z)$, from (\ref{Hiypsi1}), (\ref{Hiypsi2}) and (\ref{Hiypsi3}) it follows that $X^h(z)=\sum_{i=1}^{2n}x^i(z)H_i(z)$.

\end{rem}

\begin{prop}\label{corchete de Lie} For  $1\leq i,j,l,m\leq n$ let $R_{ijlm}:N\longrightarrow \R$ be the maps  defined by
$R_{ijlm}(q,u,\xi)=g(R(u_i,u_j)u_l,u_m)$, where $R$ is the curvature tensor of $(M,g)$. The Lie bracket on vertical and horizontal vector field on $N$ satisfies:

\begin{itemize}

  \item [a)]  $ [H_i,H_j]=
\sum_{l,m=1}^n
  R_{ijlm}\xi^mH_{n+l}+\frac{1}{2}\sum_{l,m=1}^nR_{ijlm}V^l_m.$

  \item [b)]  $ [H_i,H_{n+j}]=0.$
  \item[c)]  $ [H_i, V^l_m]=\delta_{il} H_m -\delta_{im} H_l.$
  \item [d)] $ [H_{n+i},H_{n+j}]=0.$
  \item[e)]  $ [H_{n+i},V^l_m]=\delta_{il}H_{n+m}-\delta_{im}H_{n+l}.$
  \item [f)] $ [V^i_j,V^l_m]=\delta_{il}V_{mj}+   \delta_{jl}V_{im}+  \delta_{im}V_{jl}+
  \delta_{jm}V_{li}.$
  \item [g)]   If $f:N\longrightarrow \R$ is a function that depends only on  the parameter $\xi$, then $H_i(f)=0$ and $ V^i_j(f)=\xi^iH_{n+j}(f)-\xi^j H_{n+i}(f)$.
  \item [h)]   If  $X,Y\in\ \chi(TM)$  and $v=\psi(q,u,\xi)$
  then $[X^h,Y^h]^v|_{(q,u,\xi)}=$\linebreak$\sum_{1\leq l<m\leq n}
g_q(R(\pi_*(X(v)),\pi_*(Y(v)))u_l,u_m)V^l_m(q,u,\xi)$.
\end{itemize}
\end{prop}
The proof is straightforward and follows by taking local coordinates in $M$ and the induced one in $TM$ and evaluating the forms $\theta^i$, $\theta^{n+i}$, $\omega^i_j$  on the fields $[H_r,H_s]$, $[H_r,V^l_m]$ and $[V^l_m,V^{l'}_{m'}]$ for $1\leq r,s\leq 2n$,  $1\leq l< m \leq n$  and
$1\leq l'< m' \leq n$.


\subsection{The main result.}

From now on, let $\bar{R}$ and $R^*$ be the curvature tensors of $(TM,G)$ and $(N,G^*)$. For simplicity we  denote by $<\ ,\ >$ the metrics $G$ and $G^*$. Since $\psi:(N,G^*)\longrightarrow(TM,G)$ is a Riemannian submersion, by the O'Neill formula (see \cite{ON}) we have that

\begin{eqnarray}
<\bar{R}(X,Y)Z,W>\circ\ \psi & = <R^*(X^h,Y^h)Z^h,W^h>+\frac{1}{4}< [Y^h,Z^h]^v,[X^h,W^h]^v
>\nonumber \\
& \label{formuladeOneill}\\
&  -\frac{1}{4}< [X^h,Z^h]^v,[Y^h,W^h]^v > -\frac{1}{2}<
[Z^h,W^h]^v,[X^h,Y^h]^v> \nonumber
\end{eqnarray}

If $\ Y^h(z)=\sum_{i=1}^{2n}y^j(z)H_i(z)$, $\ Z^h(z)=\sum_{i=1}^{2n}z^k(z)H_i(z)$ and $W^h(z)=\sum_{i=1}^{2n}w^l(z)H_i(z)$, then the first term of the right side of equality (\ref{formuladeOneill})  is $$<R^*(X^h,Y^h)Z^h,W^h>\ =\sum_{ijkl=1}^{2n}
x^iy^jz^kw^l<R^*(H_i,H_j)H_k,H_l>$$ On the other hand, if $v=\psi(q,u,\xi)$,  it follows from  Proposition \ref{corchete de Lie} (part h) that

$< [X^h,Y^h]^v,[Z^h,W^h]^v>|_{(q,u,\xi)}=$
\begin{equation}\label{23}
=\frac{1}{2}\sum_{r,s=1}^n<R(\pi_*(X(v)),\pi_*(Y(v)))u_r,u_s>.
<R(\pi_*(Z(v)),\pi_*(W(v)))u_r,u_s>
\end{equation}

\begin{rem}\label{z=quxi} In  order to compute  $<\bar{R}(X(v),Y(v))Z(v),W(v)>$ it is sufficient  to evaluate the right side of
(\ref{formuladeOneill}) on points of $N$ of the form $z=(q,u,t,0,\ldots,0)$ such that $v=\psi(z)=t.v$ and $t=\|v\|$.
\end{rem}


Let $f:[0,+\infty)\longrightarrow \R$ be a differentiable map, from now on,  let us denote by $\dot{f}(t)$ the derivate of $f$ at $t$.

\begin{thm}\label{Ecuaciones de Curvatura}
Let $G$ be a natural metric on $TM$, and  $\alpha$, $\beta$ be the functions that characterizes $G$. If  $1\leq i,j,k,l \leq n$ and $z=(q,u,t,0,\dots,0)$ we have that

\begin{itemize}

 \item[a)]  $<R^*(H_i(z),H_j(z))H_k(z), H_l(z))>=$
$$t^2\alpha(t^2). \sum_{r=1}^n \Big\{
\frac{1}{2}R_{ijr1}(z)R_{klr1}(z)+\frac{1}{4}R_{ilr1}(z)R_{kjr1}(z)+
\frac{1}{4}R_{jlr1}(z)R_{ikr1}(z)\Big\}$$
$$+\sum_{1\leq r<s\leq n} \Big\{
\frac{1}{2}R_{ijr1}(z)R_{klrs}(z)+\frac{1}{4}R_{ilr1}(z)R_{kjrs}(z)+\frac{1}{4}
R_{jlr1}(z)R_{ikrs}(z)\ \Big\}+R_{ijkl}(z).$$

\item[b)] Let $\epsilon_{ijkl}=\delta_{il}\delta_{jk}-\delta_{jl}\delta_{ik}$, then
\begin{itemize}
\item[b.1)]  If no index is
equal to one, then $$<R^*(H_{n+i}(z),H_{n+j}(z))H_{n+k}(z), H_{n+l}(z)>=\epsilon_{ijkl}F(t^2)$$
where $F:[0,+\infty)\longrightarrow\R$ is defined by
\begin{equation}\label{F}
F(t)=\frac{\alpha(t)\beta(t)-t(\dot{\alpha}(t))^2-2\alpha(t)\dot{\alpha}(t)}{\alpha(t)+t\beta(t)}
\end{equation}
\item[b.2)] If some index equals one, for example $l=1$, then
$$<R^*(H_{n+i}(z),H_{n+j}(z))H_{n+k}(z), H_{n+1}(z)>=\epsilon_{ijk1}H(t^2)$$
where $H:[0,+\infty)\longrightarrow\R$ is defined by
\begin{equation}\label{H}
H(t)=\phi(t)\frac{\partial }{\partial t}\ln(\alpha\Delta)|_t-2\dot{\phi}(t)
\end{equation}
and $\phi(t)=\alpha(t)+t\dot{\alpha}(t)$, $\Delta(t)=\alpha(t)+t\beta(t)$.
\end{itemize}

\item[c)]     $<R^*(H_i(z),H_{n+j}(z))H_{n+k}(z), H_{n+l}(z)>=0.$
\item[d)]    $<R^*(H_{n+i}(z),H_{n+j}(z))H_k(z), H_l(z)>=$
$$=\frac{1}{2}(2\alpha(t^2)+(\delta_{i1}+\delta_{j1})\beta(t^2)t^2)R_{ijkl}(z)+\frac{1}{2}\delta_{i1}
(\beta(t^2)-2\dot{\alpha}(t^2))t^2R_{klj1}(z)$$
$$+\frac{1}{2}\delta_{j1}
(2\dot{\alpha}(t^2)-\beta(t^2))t^2R_{kli1}(z)+\frac{(\alpha(t^2))^2t^2}{4}\sum_{r=1}^{n}\{R_{krj1}(z)R_{rli1}(z)-
R_{kri1}(z)R_{rlj1}(z)\}.$$
\item[e)]  $<R^*(H_{i}(z),H_{n+j}(z))H_k(z), H_{n+l}(z)>=$
$$\frac{1}{2}\alpha(t^2)R_{kilj}(z)+\frac{(\alpha(t^2))^2t^2}{4}\sum_{r=1}^{n}R_{krj1}(z)R_{ril1}(z)
+\frac{t^2}{2}(\delta_{j1}+\delta_{l1})\dot{\alpha}(t^2)
(R_{kil1}(z)-R_{kij1}(z)).$$
\item[f)]  $<R^*(H_{i}(z),H_{j}(z))H_{n+k}(z), H_l(z))>=$
$$\frac{\alpha(t^2)t}{2}\{<\nabla_DR(E_j^i(s),E_j^l(s))E_j^k(s)|_{s=0},u_1>-<\nabla_DR(E_i^j(s),E_i^l(s))E_i^k(s)|_{s=0},u_1>\}.$$

\end{itemize}
\end{thm}

The proof follows from the Koszul formula and Proposition \ref{corchete de Lie} and it involves a lot of calculation. For more details we refer the reader to \cite{GH2} pages 132-151.

\begin{thm}\label{corolecucurvatura} The curvature tensor  $\bar{R}$ evaluated on $e_i(z)$, $e_{n+i}(z)$ satisfies:

\begin{itemize}

\item[a)]
$<\bar{R}(e_i(z),e_j(z))e_k(z),e_l(z)>=$
$$t^2\alpha(t^2) \sum_{r=1}^n \{
\frac{1}{2}R_{ijr1}(z)R_{klr1}(z)+\frac{1}{4}R_{ilr1}(z)R_{kjr1}(z)+
\frac{1}{4}R_{jlr1}(z)R_{ikr1}(z)\}\  +R_{ijkl}(z).$$
\item[b)]
\begin{itemize}
\item[b.1)]  If no index is equal to one, then
\begin{equation}
<\bar{R}(e_{n+i}(z),e_{n+j}(z))e_{n+k}(z), e_{n+l}(z)>=\epsilon_{ijkl}.F(t^2)
\end{equation}
\item[b.2)] If some index equals one, for example $l=1$, then
\begin{equation}
<\bar{R}(e_{n+i}(z),e_{n+j}(z))e_{n+k}(z), e_{n+1}(z)>=\epsilon_{ijk1}.H(t^2)
\end{equation}
\end{itemize}
\item[c)]
$<\bar{R}(e_i(z),e_{n+j}(z))e_{n+k}(z), e_{n+l}(z)>=0.$
\item[d)]
$<\bar{R}(e_{n+i}(z),e_{n+j}(z))e_k(z), e_l(z)>=$
$$\frac{1}{2}\Big(2\alpha(t^2)+(\delta_{i1}+\delta_{j1})\beta(t^2)t^2\Big)R_{ijkl}(z)+\frac{1}{2}\delta_{i1}
\Big(\beta(t^2)-2\dot{\alpha}(t^2)\Big)t^2R_{klj1}(z)$$
$$+\frac{1}{2}\delta_{j1}
\Big(2\dot{\alpha}(t^2)-\beta(t^2)\Big)t^2R_{kli1}(z)+\frac{(\alpha(t^2))^2t^2}{4}\sum_{r=1}^{n}\{R_{krj1}(z)R_{rli1}(z)-
R_{kri1}(z)R_{rlj1}(z)\}$$
\item[e)]
$<\bar{R}(e_{i}(z),e_{n+j}(z))e_k(z), e_{n+l}(z)>=$
$$\frac{1}{2}\alpha(t^2)R_{kilj}(z)+\frac{(\alpha(t^2))^2t^2}{4}\sum_{r=1}^{n}R_{krj1}(z)R_{ril1}(z)
+\frac{t^2}{2}(\delta_{j1}+\delta_{l1})
\dot{\alpha}(t^2)(R_{kil1}(z)-R_{kij1}(z))$$
\item[f)]
$<\bar{R}(e_{i}(z),e_{j}(z))e_{n+k}(z), e_l(z))>=$
$$\frac{\alpha(t^2)t}{2}\{<\nabla_DR(E_j^i(s),E_j^l(s))E_j^k(s)|_{s=0},u_1>-<\nabla_DR(E_i^j(s),E_i^l(s))E_i^k(s)|_{s=0},u_1>\}$$
\end{itemize}
\end{thm}

\begin{proof}{Proof.} The proof is straightforward and follows form Theorem   \ref{Ecuaciones de Curvatura} and equality (\ref{formuladeOneill}).
\end{proof}
The functions $F$ and $H$ satisfy the following Proposition
\begin{prop}\label{sobreHyF}
Let $\alpha,\beta:[0,+\infty)\longrightarrow\R$ be  differentiable functions such that $\alpha(t)>0$ and $\alpha(t)+t\beta(t)>0$ for all $t\geq 0$. If $F$ is the zero function, then:
\begin{itemize}
\item[i)]  $\beta(t)=\frac{t(\dot{\alpha}(t))^2+2\alpha(t)\dot{\alpha}(t)}{\alpha(t)}$.
\item[ii)]  $\alpha(t)(\alpha(t)+t\beta(t))=(t\dot{\alpha}(t)+\alpha(t))^2$.
\item[iii)]  $\alpha(t)+t\dot{\alpha}(t)>0$.
\item[iv)]  $H(t)=0$ for all $t\geq 0$.
\end{itemize}
\end{prop}
\begin{proof}{Proof.} Assertion i) follows from equality (\ref{F}) and ii) is a consequence of i). Equality ii) shows that $\alpha(t)+t\dot{\alpha}(t)\neq 0$ for all $t\geq 0$, and since $\alpha(0)+0.\dot{\alpha}(0)=\alpha(0)>0$, then we  get iii). Equality ii) says that $\alpha.\Delta=\phi^2$, and assertion iii) says that $\phi>0$. Therefore, from equality (\ref{H}) we get that $H=0$.
\end{proof}

\begin{cor}\label{h=0} Let $\alpha,\beta:[0,+\infty)\longrightarrow\R$ be differentiable functions such that $\alpha(t)>0$, $\alpha(t)+t\dot{\alpha}(t)>0$ and $\alpha(t)+t\beta(t)>0$ if $t\geq 0$. If $H$ is the zero function, then it is also $F$.
\end{cor}

\begin{proof}{Proof.} Since $\phi>0$ and $H=0$, the equality (\ref{H}) implies that $\ln(\alpha\Delta)=\ln(\phi^2)+C$ for some constant $C$. In particular $2\ln(\alpha(0))=2\ln(\alpha(0))+C$, hence $C=0$. Since $\alpha.\Delta=\phi^2$, we obtain that $F=0$.

\end{proof}
\section{Geometric consequences of curvature equations.}\label{section 4}

In this section the Riemannian metric $G$ on $TM$ is assumed  natural. As trough all the paper, $G$ is characterized by the functions $\alpha$ and $\beta$. As in Remark \ref{z=quxi}, if $v\in TM$, let $z=(q,u,t,0,\dots,0)\in N$ such that $\psi(z)=v$ and $t=\|v\|$.
 From Theorem \ref{corolecucurvatura} and Proposition \ref{sobreHyF} we get inmediatly
 \begin{cor} If $(TM,G)$ is flat then $(M,G)$ is flat.
  \end{cor}
 \begin{proof}{Proof.} It follows from part a) of Theorem \ref{corolecucurvatura} by setting $t=0$.
 \end{proof}
 \begin{cor}If $\dim M\geq 3$, $(TM,G)$ is flat if and only if $(M,g)$ is flat and $$\beta(t)=
\frac{t(\dot{\alpha}(t))^2+2\alpha(t)\dot{\alpha}(t)}{\alpha(t)}$$

 \end{cor}
 \begin{proof}{Proof.} Assume that $(TM,G)$ is flat. From Theorem \ref{corolecucurvatura} part b.1) and $1<i<j\leq n$ we have that $$<\bar{R}(e_{n+i}(z),e_{n+j}(z))e_{n+i}(z), e_{n+j}(z)>=-F(t^2)$$
 Therefore $F=0$, and the desired equality on $\beta$ follows from Proposition \ref{sobreHyF} part i).

 Assuming that $(M,g)$ is flat and $\beta(t)=
\frac{t(\dot{\alpha}(t))^2+2\alpha(t)\dot{\alpha}(t)}{\alpha(t)}$, we only need to show that \begin{equation}\label{curvaturacerovertical}<\bar{R}(e_{n+i}(z),e_{n+j}(z))e_{n+k}(z), e_{n+l}(z)>=0\end{equation} for $1\leq i,j,k,l\leq 2n$. The other cases also satisfies (\ref{curvaturacerovertical}) because $R=0$. Equality on $\beta$ implies that $F=0$, therefore by Proposition \ref{sobreHyF} part iv) we have that $H=0$, and  equality (\ref{curvaturacerovertical}) is satisfied.
 \end{proof}

We have also immediately the following result

 \begin{cor} If $\dim M=2$, $(TM,G)$ is flat if and only if $(M,g)$ is flat and $H=0$.

  \end{cor}
 \begin{rem} Let $\alpha(t)>0$ be a differentiable function that satisfies $t\dot{\alpha}(t)+\alpha(t)>0$ for all $t\geq 0$ and define  $\beta(t)=\frac{t(\dot{\alpha}(t))^2+2\alpha(t)\dot{\alpha}(t)}{\alpha(t)}$. If we consider the natural metric $G$ induced by $\alpha$ and $\beta$, then $(TM,G)$ is flat if $(M,g)$ is flat.

 \end{rem}

 \begin{rem} The above Corollaries generalizes the well known fact that $(TM,G_s)$ is flat if and only if $(M,g)$ if flat (Kowalski \cite{Ko}, Aso \cite{KA}). This fact, follows from the Corollaries taking  $\alpha=1$ and $\beta=0$.

 \end{rem}

 We will denote by $K$ and  $\bar{K}$ the sectional curvatures of $(M,g)$ and $(TM,G)$ respectively.

\begin{thm}\label{propcurvaturaseccional} We have the following expression for the sectional curvature of  $(TM,G)$, where $z=(q,u,t,0,\dots,0)$ and $\psi(z)=v$ with $t=\|v\|$:
\begin{itemize}
\item[a)]  For $1\leq i,j\leq n$:
$$\bar{K}(e_i(z),e_j(z))=K(u_i,u_j)-\frac{3}{4}\alpha(t^2)|R(u_i,u_j)v|^2$$

\item[b)]
\begin{itemize}
\item[b.1)]  If $2\leq i, j \leq n$ and $i\neq j$ $$\bar{K}(e_{n+i}(z),e_{n+j}(z))=\frac{F(t^2)}{(\alpha(t^2))^2}$$
\item[b.2)]  If $2\leq i\leq n$ $$\bar{K}(e_{n+1}(z),e_{n+j}(z))=\frac{H(t^2)}{\alpha(t^2)(\alpha(t^2)+t^2\beta(t^2))}$$
\end{itemize}
  \item[c)]  For $1\leq i,j\leq n$: $$\bar{K}(e_i(z),e_{n+j}(z))=\frac{\alpha(t^2)}{4}|R(u_j,v)u_i|^2$$
\end{itemize}
\end{thm}
 In particular $\bar{K}(e_i,e_{n+1})=0$ if $1\leq i\leq n$, since $v=tu_1$.
\begin{proof}{Proof.}  From equality (\ref{caracterizacionmetricanatural}) we get that $e_1(z),\dots,e_{2n}(z)$ is an orthogonal basis for $(TM)_v$ such that $<e_i(z),e_j(z)>=\delta_{ij}$ if $1\leq i,j\leq n$, $<e_{n+1}(z),e_{n+1}(z)>=\alpha(t^2)+t^2\beta(t^2)$ and  $<e_{n+i}(z),e_{n+i}(z)>=\alpha(t^2)$ if $2\leq i\leq n$.
Let $1\leq i,j\leq n$, $i\neq j$. By setting $k=j$ and $l=i$ in equation a) of  Theorem \ref{corolecucurvatura} we have that
$$\bar{K}(e_i(z),e_j(z))=-<\bar{R}(e_i(z),e_j(z))e_j(z),e_i(z)>=R_{ijji}(z)-\frac{3}{4}t^2\alpha(t^2)\dst\sum_{r=1}^nR_{ij1r}^2(z)$$
Since $K(u_i,u_j)=R_{ijji}(z)$ and $v=tu_1$, we can write $$\bar{K}(e_i(z),e_j(z))=K(u_i,u_j)-\frac{3}{4}\alpha(t^2)|R(u_i,u_j)v|^2$$
Part b) follows directly from equations b.1) and b.2) of Theorem \ref{corolecucurvatura}.

 Since $\|e_i(z)\|=1$ and $<e_i(z),e_{n+j}(z)>=0$ for $1\leq i,j\leq n$, from Theorem \ref{corolecucurvatura} equation e), we see that
$$\bar{K}(e_i(z),e_{n+j}(z))=-\frac{(\alpha(|v|^2))^2|v|^2}{4(\alpha(|v|^2)+\delta_{j1}\beta(|v|^2)|v|^2)}\sum_{r=1}^{n}R_{irj1}(z)R_{rij1}(z)$$
$$=\frac{\alpha(|v|^2)}{4}\sum_{r=1}^n\Big[ g(R(u_j,u_1|v|)u_i,u_r)\Big]^2=\frac{\alpha(|v|^2)}{4}|R(u_j,v)u_i|^2.$$
\end{proof}

\begin{cor}
$\;$
\begin{itemize}

\item[i)]  $(TM,G)$ is never a manifold with negative sectional curvature.
\item[ii)]  If $\bar{K}$ is constant, then $(TM,G)$ and $(M,g)$ are flat.
\item[iii)] If $\bar{K}$ is bounded and $\lim_{t\to +\infty}t\alpha(t)=+\infty$, then $(M,g)$ is flat.
\item[iv)]  If  $c\leq\bar{K}\leq C$ (possibly $c=-\infty$ and $C=+\infty$), then   $c\leq K \leq C$.
\end{itemize}
\end{cor}
\begin{proof}{Proof.} Assertions i), ii) and ii) follow from  Theorem \ref{propcurvaturaseccional} part c).  Let $q\in M$ and $u=(u_1,\dots, u_n)$ be an orthonormal basis for $M_q$. Then, if we consider $z=(q,u,0,\dots,0)$ and $v=0_q$, from    Theorem \ref{propcurvaturaseccional} part a) we have that $\bar{K}(e_i(z),e_j(z))=K(u_i,u_j)$ and  part iv) holds. Also ii) follows from   Theorem \ref{corolecucurvatura}) part a) taking $t=0$.

\end{proof}

\begin{cor}\label{curvaturasecccurvaturaconstante} Let  $(M,g)$ be a manifold of constant sectional curvature $K_0$ and $TM$ endowed with a natural metric $G$, then we have  for $z=(q,u,t,0,\dots,0)$ and $\psi(z)=v$ that
\begin{itemize}
\item[] a) $\bar{K}(e_i(z),e_j(z))=K_0-\frac{3}{4}(K_0)^2\alpha(|v|^2)(\delta_{i1}+\delta_{j1})|v|^2$ with $i\neq j$.
\item[] b) $\bar{K}(e_i(z),e_{n+j}(z))=\frac{\alpha(|v|^2)}{4}K_0|v|^2(\delta_{ij}+\delta_{i1})$.
\end{itemize}
 The vertical case $\bar{K}(e_{n+i},e_{n+j})$ is as   Theorem \ref{propcurvaturaseccional} part b).
\end{cor}

From  Theorem \ref{propcurvaturaseccional} we get the following result
\begin{cor} Let  $G_1$ and $G_2$ be two natural metrics on $TM$ such that are characterized by the functions  $\{\alpha_i\}_{i=1,2}$ and $\{\beta_i\}_{i=1,2}$. If  $\bar{K}_1(u)(V,W)=\bar{K}_2(u)(V,W)$ for all $u\in TM$ and $V,W\in (TM)_u$ and  $(M,g)$ is not flat, then $\alpha_1=\alpha_2$.
\end{cor}

\begin{rem}\label{metricaexponecial}
Let $G_{+\exp}$ and $G_{-\exp}$ be the natural metrics on  $TM$ defined by

\begin{eqnarray*}
^{g}G_{+\exp}(q,u,\xi)=\pmatrix{Id_{n\times n} & 0 \cr 0 & A^+(\xi)}\  \ \ \ \mbox{and}  \ \ \ \ \ ^{g}G_{-\exp}(q,u,\xi)=\pmatrix{Id_{n\times n} & 0 \cr 0 & A^-(\xi)}
\end{eqnarray*}

where $A^+(\xi)=e^{|\xi|^2}(Id_{n\times n}+\xi^t.\xi)$ and $A^-(\xi)=e^{-|\xi|^2}(Id_{n\times n}+\xi^t.\xi)$. We call $G_{+\exp}$ and $G_{-\exp}$ the positive and negative exponential metric.

It is known (\cite{Se}) that $TM$ endowed with the Cheeger-Gromoll metric is never a manifold of constant sectional curvature. Theorem \ref{propcurvaturaseccional} applied to $G_{+\exp}$ and $G_{-\exp}$ shows that these metrics satisfy the same property.

\end{rem}


\subsection{Ricci tensor and scalar curvature.}

Let  $Ricc$ and  $\bar{R}icc$  be  the Ricci tensor of  $(M,g)$ and $(TM,G)$ respectively. We will denote by $S$ and  $\bar{S}$ the scalar curvature of $(M,g)$ and $(TM,G)$.
\begin{thm}\label{thmRicc}  For $1\leq i,j\leq n$ and $z=(q,u,t,0\dots,0)$  we have the following expressions for  $\bar{R}icc$:
\begin{itemize}
\item[a)]  $\bar{R}icc(e_i(z),e_j(z))=-\frac{\alpha(t^2)t^2}{2}\dst\sum_{1\leq r,l\leq n}R_{irl1}(z)R_{jrl1}(z)+Ricc(u_i,u_j)$
\item[b)]  $\bar{R}icc(e_{i}(z),e_{n+j}(z))=-\frac{\alpha(t^2)t^2}{2}\dst\sum_{1\leq r\leq n}\Big\{<\nabla_DR(E^i_r,E^r_r)E^j_r|_{s=0},u_1>$

     $\ \ \ \ \ \ \ \ \ \  \ \ \ \ \ \ \ \ \ \ \ \ \ \ \ \ \ -<\nabla_DR(E^r_i,E^r_i)E^j_i|_{s=0},u_1> \Big\}$
\item[c)]
\begin{itemize}
\item[c.1)]  If $2\leq i\leq n$, then
$$\bar{R}icc(e_{n+i}(z),e_{n+i}(z))=\frac{t^2\alpha(t^2)}{4}\dst\sum_{1\leq r,l\leq n}R_{rli1}^2(z)+\frac{(n-2)}{\alpha(t^2)}F(t^2)$$$$\ \ \ \ \ \  \ \ \ \ \ \ \ +\frac{1}{\alpha(t^2)+t^2\beta(t^2)}H(t^2)$$
\item[c.2)]  If $2\leq i,j \leq n$ and $i\neq j$, then
$$\bar{R}icc(e_{n+i}(z),e_{n+j}(z))=\frac{t^2\alpha(t^2)}{4}\dst\sum_{1\leq r,l\leq n}R_{rli1}(z)R_{rlj1}(z)$$
\item[c.3)]  If $1\leq j\leq n$, then
$$\bar{R}icc(e_{n+1}(z),e_{n+j}(z))=\frac{(n-1)}{\alpha(t^2)}H(t^2)\delta_{j1}$$
\end{itemize}

\end{itemize}
\end{thm}
\begin{proof}{Proof.} Let $\bar{e}_1(z),\dots,\bar{e}_{2n}(z)$ be the orthonormal basis for $(TM)_v$ induced by the orthogonal basis $e_1(z),\dots,e_{2n}(z)$, where $\psi(z)=v$. For $X, Y\in (TM)_v$ we have that $$\bar{R}icc(X,Y)=\sum_{l=1}^{2n}<\bar{R}(X,\bar{e}_l(z))\bar{e}_l(z),Y>$$
Equalities a), b) and c) follow directly from Theorem \ref{corolecucurvatura} and the fact that \linebreak$<e_{n+1}(z),e_{n+1}(z)>=\alpha(t^2)+t^2\beta(t^2)$ and $<e_{n+i}(z),e_{n+i}(z)>=\alpha(t^2)$ if $2\leq i\leq  n$.
\end{proof}

\begin{cor} Let $\alpha$ and $\beta$ be the functions that characterizes $G$, such that $\alpha(t)+t\dot{\alpha}(t)>0$ for $t\geq0$. If $(TM,G)$ is Ricci flat then $(M,g)$ and $(TM,G)$ are flats.
\end{cor}
\begin{proof}{Proof.}  In order to prove that $R=0$, it is enough to show that for any $q\in M$ and any orthonormal basis $u=\{u_1,\dots,u_n\}$ for $M_q$ the following equalities are satisfied
\begin{equation}\label{e}<R(u_r,u_l)u_i,u_1>=0
\end{equation}  for $1\leq r,l\leq n$ and $2\leq i\leq n$.
Let $v\in M_q$, $v\neq 0$ and $z=(q,u,t,0,\dots,0)\in N$ such that $\psi(z)=tu_1=v$. If $\bar{R}icc=0$, from Theorem \ref{thmRicc} part c.3) we have that $H=0$. Since   $\alpha(t)+t\dot{\alpha}(t)>0$, we get from Corollary \ref{h=0} that $F=0$. Consequently, equalities (\ref{e}) follows from c.1). Since $R=0$ and $H=F=0$, from Theorem \ref{corolecucurvatura} we have that $\bar{R}=0$.

\end{proof}

\begin{rem}  Is easy  to see from Theorem \ref{thmRicc} that  if  $(M,g)$ is not flat or if not  exists a constant $k$ such that   $H(t)=k\alpha(t)$ and $(n-2)[\alpha(t)+t\beta(t)]F(t)=\alpha(t)k\Big[(n-2)\alpha(t)+(n-1)t\beta(t)\Big]$, then
$\bar{R}icc$ is not a $\lambda-natural\ tensor$ (see \cite{GH}).
\end{rem}

\begin{cor}\label{corcurvaturascalar}  Let $v\in TM$ and  $z=(\pi(v),u_1,\dots,u_n,t,0\dots,0)\in N$ such that $v=u_1t$.  The scalar curvature of $(TM,G)$ at $v$ is given by
$$\bar{S}(v)=S(\pi(v))-\frac{t^2\alpha(t^2)}{4}\dst\sum_{irl=1}^nR_{irl1}^2(z)+\frac{2(n-1)}{\alpha(t^2)\Big(\alpha(t^2)+\beta(t^2)t^2\Big)}H(t^2)$$
$$+\frac{(n-1)(n-2)}{(\alpha(t^2))^2}F(t^2)$$
\end{cor}
\begin{proof}{Proof.} Since $\{\bar{e}_1(z),\dots,\bar{e}_{2n}(z)\}$ is an orthonormal basis for $(TM)_v$ and the scalar curvature $\bar{S}(v)=\sum_{l=1}^{2n}Ricc(\bar{e}_l(z),\bar{e}_l(z))$, the expression  for $\bar{S}$ follows straightforward from Theorem \ref{thmRicc}.

\end{proof}


\begin{rem}
 Corollary \ref{corcurvaturascalar} applied to  $G_{+\exp}$ and $G_{-\exp}$ reads:
$$S_{+\exp}(v)=S(\pi(v))-(n-1)e^{-|v|^2}\frac{\Big[2+(n-2)(1+|v|^2)\Big]}{(1+|v|^2)}$$$$-\frac{e^{|v|^2}}{4}\sum_{i,j=1}^n|R(u_i,u_j)v|^2   $$
 $$ S_{-\exp}(v)=S(\pi(v))+
 \frac{(n-1)e^{|v|^2}}{1+|v|^2}\Big [(n-2)(3-|v|^2) +\frac{6+2|v|^2}{1+|v|^2}\Big ]$$ $$-\frac{e^{-|v|^2}}{4}\sum_{i,j=1}^n|R(u_i,u_j)v|^2 $$
\end{rem}
\begin{prop}\label{curvaturaescalarexp} If $(M,g)$ is a manifold of constant sectional curvature $K_0$, then
$$S_{+\exp}(v)=(n-1)\Big\{K_0\Big( n- \frac{K_0}{2}|v|^2e^{|v|^2}\Big) -e^{-|v|^2}\frac{\Big[2+(n-2)(1+|v|^2)\Big]}{(1+|v|^2)}                      \Big\}$$
$$S_{-\exp}(v)=(n-1)\Big\{K_0\Big( n- \frac{K_0}{2}|v|^2e^{-|v|^2}\Big)+\frac{e^{|v|^2}}{1+|v|^2}\Big [(n-2)(3-|v|^2) +\frac{6+2|v|^2}{1+|v|^2}\Big] \Big\}                       $$
\end{prop}
\begin{cor} Let $(M,g)$ be a flat manifold, then we have that:
\begin{itemize}
\item[] a) $S_{+\exp}<0$.
\item[] b) If $\dim M=2$, then $S_{-\exp}>0$.
\item[] c) If $\dim\geq 3$,   $S_{\exp}(v)> 0$ if and only if $0\leq |v|^2 < \frac{(n-1)+\sqrt{4(n-2)n+1}}{n-2}$.
\item[] d) Si $\dim\geq 3$,  $S_{\exp}(v)=0$ if and only if $|v|^2=\frac{(n-1)+\sqrt{4(n-2)n+1}}{n-2}$.
\end{itemize}
\end{cor}
\begin{proof}{Proof.} It follows from Proposition \ref{curvaturaescalarexp}.

\end{proof}



\noindent{\sc Guillermo Henry:   \\
Departamento de Matem\'atica,  FCEyN, Universidad de Buenos Aires\\
 Ciudad Universitaria, Pabell\'on I,  Buenos Aires, C1428EHA, Argentina  }\\
{\it e-mail address}: ghenry@dm.uba.ar  \\

\noindent{\sc Guillermo Keilhauer  \\
Departamento de Matem\'atica,  FCEyN, Universidad de Buenos Aires\\
 Ciudad Universitaria, Pabell\'on I,  Buenos Aires, C1428EHA, Argentina  }\\
{\it e-mail address}: wkeilh@dm.uba.ar  \\

\end{document}